\documentclass[preprint,12pt]{elsarticle}

\usepackage{amssymb}
\usepackage{amsthm}
\usepackage{amsmath,amssymb}
\usepackage{graphicx}
\usepackage{booktabs}
\usepackage{algorithm}
\usepackage{algorithmicx}
\usepackage{algpseudocode}
\usepackage{multirow}
\usepackage{float}
\usepackage{amsmath}

\newtheorem{definition}{Definition}%
\newtheorem{theorem}{Theorem}

\newtheorem{proposition}[theorem]{Proposition}

\newtheorem{remark}[theorem]{Remark}

\journal{Journal of Computational and Applied Mathematics}

\begin{document}

\begin{frontmatter}



\title{Recursive algorithms for computing Birkhoff interpolation polynomials}


\author{Xue Jiang}
\author{Yuanhe Li}
\author{Zhe Li\corref{cor1}} 
\cortext[cor1]{Corresponding author}
\ead{zheli200809@163.com}
\affiliation{organization={School of Mathematics and Statistics, Changchun University of Science and Technology},
            city={Changchun},
            postcode={130000},
            country={China}}
\begin{abstract}
As a generalization of Hermite interpolation problem, Birkhoff interpolation is an important subject in numerical approximation. This paper generalizes the existing Generalized Recursive Polynomial Interpolation Algorithm (GRPIA) that is used to compute the Hermite interpolation polynomial. Based on the theory of the Schur complement and the Sylvester identity, the proposed recursive algorithms are applicable to a broader class of Birkhoff interpolation problems, where each interpolation condition is given by the composition of an evaluation functional and a differential polynomial. The approach incorporates a judgment condition to ensure the problem's well-posedness and computes a lower-degree Newton-type interpolation basis (which is also a strongly proper interpolation basis) along with the corresponding interpolation polynomial. Following the numerical examples, we analyze and compare the computational process and complexity of the proposed algorithm against traditional interpolation methods based on Gaussian elimination, and thus demonstrate that the proposed recursive approach reduces both computational cost and storage space requirements.
\end{abstract}




\end{frontmatter}



\section{Introduction}





Polynomial interpolation is of significant importance in numerous practical engineering fields and continues to motivate theoretical research. Generally, polynomial interpolation problems fall into three categories: Lagrange interpolation, Hermite interpolation, and Birkhoff interpolation. For the problems of Lagrange and Hermite interpolation, all polynomials that satisfy the homogeneous interpolation conditions form an ideal, and thus they can also be referred to as ideal interpolation. In 1966, Schoenberg proposed a general definition of the Birkhoff interpolation problem \cite{wen1}, which can also be termed the non-ideal interpolation problem.

For the Hermite interpolation problem, the derivative conditions at each node are contiguous. In contrast, Birkhoff interpolation imposes no such continuity requirement on the derivatives. Thus, Birkhoff interpolation can be considered as a generalization of the Hermite interpolation problem. The univariate Hermite interpolation problem is well-established, with comprehensive solutions available for the construction of the proper interpolation space, the interpolation polynomial, and the corresponding error analysis. However, the Birkhoff interpolation problem is more complex and lacks a complete theoretical framework. As a result, even in the univariate case, there is no unified criterion for the regularity of the interpolation problem, nor a direct method to specify the form of the interpolation polynomial.

 Schoenberg proposed the theory of the so-called incidence matrix and its regularity in \cite{wen1}, after which research on the Polya's conditions has advanced rapidly. Polya's conditions provide a necessary and sufficient criterion for the well-posedness of a two-point interpolation problem. In \cite{wen2}, Tur\'{a}n investigated the (0,2)-interpolation problem and proposed several open problems related to Birkhoff interpolation. Shi presented a consolidation of the established results concerning the regularity of both the (0,2)-interpolation problem and specific Birkhoff interpolation problems \cite{wen7}. In \cite{ration}, Xia et al. investigated the linearization of univariate Birkhoff rational interpolation and demonstrated the effectiveness of these rational methods through numerical experiments. Furthermore, many researchers have investigated specific aspects of multivariate Birkhoff interpolation problems, see \cite{wen11,wen12,lei,wen14,wen15,k.of}.


Based on the theory of Schur complement and Sylvester identity, Messaoudi et al. proposed efficient algorithms for univariate Lagrange and Hermite interpolation, namely the RPIA \cite{wen19} and the GRPIA \cite{wen20} algorithms.  In \cite{wen22}, the MRPIA algorithm was developed for a specialized class of Hermite interpolation problems where nodal conditions contain only derivatives of order zero and unity. Subsequently, for the multivariate Lagrange interpolation problem where the nodes constitute a grid, the RMVPIA algorithm was developed to compute the corresponding interpolation polynomials \cite{wen21}. Furthermore, Errachid et al. \cite{wen23} reformulated the multivariate Lagrange interpolation problem as a univariate one.  This allows the method to fully leverage established univariate schemes, such as Newton interpolation or the method of divided differences.

Recent work by Rhouni et al. has addressed the Birkhoff interpolation problem where the conditions involve only derivatives at the nodes.
In \cite{wen27},  by constructing a set of high-degree Newton-type basis that can be factorized, 
the authors proposed the RHBPIA algorithm to determine the Birkhoff interpolation polynomial. 
The algorithms proposed in this paper differs from the RHBPIA in the following two aspects: 1. The proposed algorithms incorporate a core judgment condition to ensure the non-singularity of Vandermonde matrix, and then compute a lower-degree interpolation basis and the corresponding interpolation  polynomial.  
2. Algorithm 2 considers a more general Birkhoff interpolation problem in which each interpolation condition is given by the composition of an evaluation functional and a differential polynomial. 


The structure of this paper is as follows. Section 2 introduces fundamental concepts of the Birkhoff interpolation problem,  the concept of Schur complement, and the matrix  Sylvester identity. Section 3 defines the concept of a strongly proper interpolation basis and derives a recursive interpolation formula for a given Birkhoff interpolation problem. Section 4 proposes two algorithms for the univariate Birkhoff interpolation problem and provides numerical examples to demonstrate their performance. The final section provides conclusions.
\section{Preliminary}

We first introduce some relevant concepts  related to Birkhoff interpolation problem.
Throughout the paper, let $\mathbb{F}$ denote either the real field $\mathbb{R}$ or the complex field $\mathbb{C}$. $\mathbb{F}[x]$ stands for  the polynomial ring over $\mathbb{F}$. Let $x_0, x_1, \dots, x_n$ be  $n+1$ pairwise distinct interpolation nodes.  
The given $N$ interpolation conditions can be expressed by linear functionals, i.e., the composite operation of evaluation functionals and differential operators
\begin{align*}
   \begin{array}{cccc}
          \delta_{x_0}\circ D_x^{\alpha_{0,1}},& \delta_{x_0}\circ D_x^{\alpha_{0,2}},& \dots,& \delta_{x_0}\circ D_x^{\alpha_{0,s_0}}, \\
          \delta_{x_1}\circ D_x^{\alpha_{1,1}},& \delta_{x_1}\circ D_x^{\alpha_{1,2}},& \dots,& \delta_{x_1}\circ D_x^{\alpha_{1,s_1}}, \\
          \vdots & \vdots & ~ & \vdots\\
          \delta_{x_n}\circ D_x^{\alpha_{n,1}},& \delta_{x_n}\circ D_x^{\alpha_{n,2}},& \dots,& \delta_{x_n}\circ D_x^{\alpha_{n,s_n}}, \\
        \end{array} 
\end{align*}
where $D_x=\frac{d}{dx},$ and $\alpha_{i,j} \in \mathbb{N}, i=0,1,\dots,n, j=1,2,\dots,s_i.$
In this paper, we define the set of pairs
\begin{equation}\label{formula:e}
 \tilde e= \{(\beta, \alpha) \mid \beta  \in \{0,1,\dots,n\},\alpha \in \{\alpha_{\beta,1},\alpha_{\beta,2},\dots,\alpha_{\beta,s_{\beta}}\} \},
 \end{equation}
and the total number of elements in $\tilde e$ is $N$.
The Birkhoff interpolation problem is to find a polynomial subspace $\mathcal{F}\subset \mathbb{F}[x]$ such that there exists a unique $p(x)\in \mathcal{F}$ satisfying the following interpolation conditions
\begin{equation*}
  p^{(\alpha)}(x_\beta)=y_{\beta, \alpha},~~ \forall (\beta, \alpha) \in \tilde e ,  
\end{equation*}
where $y_{\beta, \alpha}\in \mathbb{F}$ are pre-specified interpolation values. \( p(x) \)  is called the interpolation polynomial,
 \( \mathcal{F} \) is called a proper interpolation space, and its dimension is equal to the number of interpolation conditions, i.e., $\dim{\mathcal{F}}=N.$

\begin{definition}\label{def:order}
Given the interpolation conditions, let the set of pairs $\tilde{e}$ be defined as in  $(\ref{formula:e})$.
We say that $  T = [(\beta_1, \alpha_1), (\beta_2, \alpha_2), \dots, (\beta_N, \alpha_N)] $
is an ordered sequence of $N-$dimensional ($N-$DOS) if\\
\noindent  $\textup{(\romannumeral 1)}$  $ \alpha_1\leq \alpha_2\leq\dots\leq \alpha_N;$\\
\noindent $\textup{(\romannumeral 2)}$ if $\alpha_s=\alpha_{s+1},$ then \ $\beta_{s}< \beta_{s+1}.$\\
Here $\beta_1, \beta_2, \dots, \beta_N $ are the node numbers \( 0, 1, \dots, n \), 
and $\alpha_i,i=1,2,\dots,N$ denotes the derivative order corresponding to the interpolation node $x_{\beta_i}$.
\end{definition}
In the following text, for a given Birkhoff interpolation problem, we always sort the interpolation conditions according to the above definition. The interpolation conditions are as follows:
\begin{align}\label{equa2.1}
   \delta_{x_{\beta_1}}\circ D_x^{\alpha_1},\delta_{x_{\beta_2}}\circ D_x^{\alpha_2},\dots,\delta_{x_{\beta_N}}\circ D_x^{\alpha_N},
\end{align}
in which $  T = [(\beta_1, \alpha_1), (\beta_2, \alpha_2), \dots, (\beta_N, \alpha_N)] $ is an $N-$DOS. For example, let the Birkhoff interpolation problem be defined by the following conditions
\begin{displaymath}
    \delta_{x_0},\delta_{x_1}\circ D_x,\delta_{x_1}\circ D_x^2,\delta_{x_2}\circ D_x^2,
\end{displaymath}
then, $  T = [(\beta_1, \alpha_1), (\beta_2, \alpha_2), (\beta_3, \alpha_3), (\beta_4, \alpha_4)] $ and $(\beta_1, \alpha_1)=(0,0),(\beta_2, \alpha_2)=(1,1), (\beta_3, \alpha_3)=(1,2), (\beta_4, \alpha_4)=(2,2)$.
\begin{definition}\cite{3M}
  If $L_i, i=1,2,...,N$ are linear functionals, a triangular sequence $(g_1,g_2,\dots,g_N)$ of polynomials for $L_1,L_2,\dots,L_N$ is a sequence such that
  $$L_{i}(g_j)=\left\{
    \begin{array}{ll}
      1, & i=j; \\
      0, & i<j.
    \end{array}
  \right.$$
\end{definition}
Note that a triangular sequence is just the Newton-type basis of the interpolation problem.
\begin{theorem}\cite{3M,82lunwen}\label{zhongzhi}
  Given linear functionals $L_i, i=1,2,...,N$, a triangular sequence $(g_1,g_2,\dots,g_N)$ exists if and only if $L_1,L_2,\dots,L_N$ are linear independent.
\end{theorem}

Next, for constructing the recursive algorithms for computing interpolation polynomials, we need to  introduce  the concept of  Schur complement and the matrix Sylvester identity, for more details, see \cite{wen24,F,G,wen26}.
\begin{definition}\label{def:schur}\cite{wen24}
Let $M$ be a matrix partitioned into four blocks
\begin{equation*}
    M=\left(
        \begin{array}{cc}
          A & B \\
          C & D \\
        \end{array}
      \right),
\end{equation*}
where the submatrix $D$ is assumed to be square and non-singular. The Schur complement of $D$ in $M$, denoted by $(M/D)$, is defined as
\begin{equation*}\label{equa2.6}
(M/D) = A - BD^{-1}C.
\end{equation*}
\end{definition}

\begin{proposition}\label{prop:rank}\cite{wen24}
Let $D$ be a non-singular matrix and $M$ be a square matrix. Then 
\begin{equation*}\label{equa2.7}
|M| = |D|\cdot|(M/D)|,
\end{equation*}
where $|\cdot|$ denotes determinant.
\end{proposition}

\begin{proposition}\label{prop:change}\cite{wen20,wen25}
~Assuming that $D$ is a non-singular matrix. Then
\begin{equation*}\label{equa2.8}
\left(\left(
   \begin{array}{cc}
     A & B \\
     C & D \\
   \end{array}
 \right)
/D\right)=\left(\left(
   \begin{array}{cc}
     D & C \\
     B & A \\
   \end{array}
 \right)
/D\right)=\left(\left(
   \begin{array}{cc}
     B & A \\
     D & C \\
   \end{array}
 \right)
/D\right)=\left(\left(
   \begin{array}{cc}
     C & D \\
     A & B \\
   \end{array}
 \right)
/D\right).
\end{equation*}
\end{proposition}

\begin{proposition}\label{prop:E}\cite{wen20,wen25}
~Assuming that $E$ is a given linear operator, $D$ is a non-singular matrix, and both $E \ast A$ and $E \ast B$ are well-defined, then we have
\begin{equation*}\label{equa2.9}
\left(\left(
   \begin{array}{cc}
     E*A & E*B \\
     C & D \\
   \end{array}
 \right)
/D\right)=E*\left(\left(
         \begin{array}{cc}
           A & B \\
           C & D \\
         \end{array}
       \right)
/D\right).
\end{equation*}
\end{proposition}

\begin{proposition}\label{prop:identity}(\cite{wen20,wen26}~{The Sylvester identity}) 
Let $M$ be the matrix defined in Definition $\ref{def:schur}$, and $K$ be a partitioned matrix as shown below
\begin{displaymath}
K=\left(
    \begin{array}{ccc}
      E & F & G \\
      H & A & B \\
      L & C & D \\
    \end{array}
  \right).
\end{displaymath}
If $A$ and $M$ are non-singular square matrices, then
\begin{align*}
(K/M)&=((K/A)/(M/A))\\
&=\left(\left(
                       \begin{array}{cc}
                         E & F \\
                         H & A \\
                       \end{array}
                     \right)
/A\right)-\left(\left(
       \begin{array}{cc}
         F & G \\
         A & B \\
       \end{array}
     \right)
/A\right)(M/A)^{-1}\left(\left(
                \begin{array}{cc}
                  H & A \\
                  L & C \\
                \end{array}
              \right)
/A\right).
\end{align*}
\end{proposition}

\section{Main results}


For the Birkhoff interpolation problem with conditions $(\ref{equa2.1})$, if we have computed a basis for the proper interpolation space $\mathcal{F}$, i.e.,
$$\mathcal{F}=\text{span}\{q_0(x), q_1(x), \dots, q_{N-1}(x) \},$$
in which 
$\deg(q_0(x))<\deg(q_1(x))<\dots<\deg(q_{N-1}(x))$, then the interpolation polynomial $p(x)\in \mathbb{F}[x]$ can be expressed as
\begin{equation*}\label{equa2.2}
p(x)=\sum_{i=0}^{N-1}a_iq_i(x),
\end{equation*}
where  $\boldsymbol{a} =(a_0 ,a_1, \dots,a_{N-1})^\textup{T} $ is the solution to the following linear system,
\begin{equation}\label{equa2.3}
V_{N-1}\boldsymbol{a} =Y_N,
\end{equation}
with $Y_N=(y_{\beta_1,\alpha_1},y_{\beta_2,\alpha_2},\dots,y_{\beta_N,\alpha_N})^\textup{T}$ and
\begin{equation}\label{equa2.4}
V_{N-1}=\left(
                \begin{array}{ccccc}
                  q_0^{(\alpha_1)}(x_{\beta_1}) & q_1^{(\alpha_1)}(x_{\beta_1}) & \dots & q_{N-1}^{(\alpha_1)}(x_{\beta_1}) \\
                  q_0^{(\alpha_2)}(x_{\beta_2}) & q_1^{(\alpha_2)}(x_{\beta_2}) & \dots & q_{N-1}^{(\alpha_2)}(x_{\beta_2}) \\
                  \vdots & \vdots & ~& \vdots  \\
                  q_0^{(\alpha_N)}(x_{\beta_N}) & q_1^{(\alpha_N)}(x_{\beta_N}) & \dots & q_{N-1}^{(\alpha_N)}(x_{\beta_N}) \\
                \end{array}
              \right).
\end{equation}
According to the uniqueness of $p(x)$, it follows that $V_{N-1}$ is a non-singular matrix. Conversely, for an $N$-dimensional polynomial subspace, still denoted by $\mathcal{F}=\textup{span}\{q_0(x),q_1(x),\dots,q_{N-1}(x)\}$, if  the corresponding matrix $V_{N-1}$ in Equation (\ref{equa2.4}) is non-singular, then $\mathcal{F}$ is a  proper interpolation space for the Birkhoff interpolation problem, and the interpolation polynomial can be obtained by solving Equations (\ref{equa2.3}) and (\ref{equa2.4}). That is,
\begin{equation*}
p(x)=(q_0(x),q_1(x),\dots,q_{N-1}(x))\boldsymbol{a}=(q_0(x),q_1(x),\dots,q_{N-1}(x))V_{N-1}^{-1}Y_N.
\end{equation*}

\begin{definition}\label{def:8}
Suppose that 
$\{q_0(x), q_1(x),\dots,q_{N-1}(x)\}$ is a proper interpolation basis for the Birkhoff interpolation problem (\ref{equa2.1}).  If for each $i=1,2,\dots,N$, $\{q_0(x), q_1(x),\dots,q_{i-1}(x)\}$ forms a proper interpolation basis for the first $i$ conditions, then  $\{q_0(x), q_1(x),\dots,q_{N-1}(x)\}$ is called a strongly proper interpolation basis.
\end{definition}
It should be noted that $\mathcal{F}=\textup{span}\{1,x,\dots,x^{N-1}\} $ is always a strongly proper  interpolation space for Hermite interpolation problems, but this is  not  true for Birkhoff interpolation problems.

For a given Birkhoff interpolation problem  ($\ref{equa2.1}$), 
the interpolation polynomial $p_{N-1}(x)$ can be written as 
\begin{align*}
p_{N-1}(x)&=0+(q_0(x),q_1(x),\dots,q_{N-1}(x))V_{N-1}^{-1}Y_N \\
&=-\left(\left[
              \begin{array}{c|cccc}
                0 & q_0(x) & q_1(x) & \dots & q_{N-1}(x) \\ \hline
                y_{\beta_1, \alpha_1} & q_0^{(\alpha_1)}(x_{\beta_1}) & q_1^{(\alpha_1)}(x_{\beta_1}) & \dots & q_{N-1}^{(\alpha_1)}(x_{\beta_1}) \\
                y_{\beta_2, \alpha_2} & q_0^{(\alpha_2)}(x_{\beta_2}) & q_1^{(\alpha_2)}(x_{\beta_2}) & \dots & q_{N-1}^{(\alpha_2)}(x_{\beta_2}) \\
                \vdots & \vdots & \vdots & ~ & \vdots \\
                y_{\beta_N, \alpha_N} & q_0^{(\alpha_N)}(x_{\beta_N}) & q_1^{(\alpha_N)}(x_{\beta_N}) & \dots & q_{N-1}^{(\alpha_N)}(x_{\beta_N}) \\
              \end{array}
           \right]/V_{N-1} \right),
\end{align*}
where $\{q_0(x),q_1(x),\dots,q_{N-1}(x) \}$ is a strongly proper interpolation basis. This expresses $p_{N-1}(x)$
as the Schur complement of a matrix, where in Definition $\ref{def:schur}$, $A=0$ is a scalar.

We utilize the properties of the Schur complement and the Sylvester identity to compute $p_{N-1}(x)$. Let $V_{k-1}, k=1,2,\dots,N$ denote the first $k$ leading principal minors of $V_{N-1}$, i.e.,
\begin{equation*}
V_{k-1}:=\left(
                \begin{array}{ccccc}
                  q_0^{(\alpha_1)}(x_{\beta_1}) & q_1^{(\alpha_1)}(x_{\beta_1}) & \dots & q_{k-1}^{(\alpha_1)}(x_{\beta_1}) \\
                  q_0^{(\alpha_2)}(x_{\beta_2}) & q_1^{(\alpha_2)}(x_{\beta_2}) & \dots & q_{k-1}^{(\alpha_2)}(x_{\beta_2}) \\
                  \vdots & \vdots & ~ & \vdots \\
                  q_0^{(\alpha_k)}(x_{\beta_k}) & q_1^{(\alpha_k)}(x_{\beta_k}) & \dots & q_{k-1}^{(\alpha_k)}(x_{\beta_k}) \\
                \end{array}
              \right).
\end{equation*}
For  $k=1,2,\dots,N$, we define the polynomial
\begin{equation}\label{equa3.2}
p_{k-1}(x):=-\left(\left[\begin{array}{c|cccc}
           0 & q_0(x) & q_1(x) & \dots & q_{k-1}(x) \\ \hline
           y_{\beta_1, \alpha_1} & q_0^{(\alpha_1)}(x_{\beta_1}) & q_1^{(\alpha_1)}(x_{\beta_1}) & \dots & q_{k-1}^{(\alpha_1)}(x_{\beta_1}) \\
            y_{\beta_2, \alpha_2} & q_0^{(\alpha_2)}(x_{\beta_2}) & q_1^{(\alpha_2)}(x_{\beta_2}) & \dots & q_{k-1}^{(\alpha_2)}(x_{\beta_2}) \\
            \vdots & \vdots & \vdots & ~ & \vdots \\
            y_{\beta_k, \alpha_k} & q_0^{(\alpha_k)}(x_{\beta_k}) & q_1^{(\alpha_k)}(x_{\beta_k}) & \dots & q_{k-1}^{(\alpha_k)}(x_{\beta_k}) \\
         \end{array}\right]/V_{k-1}\right).
\end{equation}
It is easy to see that $p_{k-1}(x)$ is the interpolation polynomial corresponding to the first $k$ interpolation conditions in  (\ref{equa2.1}).

For any fixed $k=1,2,\dots,N$, let $g_{-1,k-1}(x)=q_{k-1}(x)$, and $g_{j,k-1}(x)$, $j=0,1,\dots,k-2$ be auxiliary polynomials defined by
\begin{equation}\label{equa3.3}
g_{j,k-1}(x):= \left(\left[\begin{array}{c|cccc}
            q_{k-1}(x) & q_0(x) & q_1(x) & \dots & q_{j}(x) \\ \hline
            q_{k-1}^{(\alpha_1)}(x_{\beta_1}) & q_0^{(\alpha_1)}(x_{\beta_1}) & q_1^{(\alpha_1)}(x_{\beta_1}) & \dots & q_{j}^{(\alpha_1)}(x_{\beta_1}) \\
            q_{k-1}^{(\alpha_2)}(x_{\beta_2}) & q_0^{(\alpha_2)}(x_{\beta_2}) & q_1^{(\alpha_2)}(x_{\beta_2}) & \dots & q_{j}^{(\alpha_2)}(x_{\beta_2}) \\
            \vdots & \vdots & \vdots & ~ & \vdots \\
            q_{k-1}^{(\alpha_{j+1})}(x_{\beta_{j+1}}) & q_0^{(\alpha_{j+1})}(x_{\beta_{j+1}}) & q_1^{(\alpha_{j+1})}(x_{\beta_{j+1}}) & \dots & q_{j}^{(\alpha_{j+1})}(x_{\beta_{j+1}}) \\
            \end{array}\right]/V_{j}\right).
\end{equation}

Inspired by \cite{wen19,wen20},  we present analogous results for our problem and include proofs following its idea for completeness.

\begin{proposition}\label{prop:V.V}
Let $\textup{span}\{q_0(x), q_1(x),\dots,q_{N-1}(x)\}$ be a strongly proper interpolation space for the Birkhoff interpolation problem $(\ref{equa2.1})$. For any fixed $k=1,2,\dots,N$, let $g_{-1,k-1}(x)=q_{k-1}(x)$. Then  
\begin{align*}
    g_{k-2,k-1}^{(\alpha_i)}(x_{\beta_i})=\left\{
    \begin{array}{ll}
      0, & i=1,2,\dots,k-1; \\
      \dfrac{|V_{k-2}|}{|V_{k-1}|}\neq 0, & i=k.
    \end{array}
  \right.
\end{align*}
\end{proposition}

\newproof{pf}{proof}
\begin{proof}
For any fixed $k=1,2,\dots,N$, we first consider the case $i=1,2,\dots,k-1$. By the definition of $g_{j,k-1}(x)$ in $(\ref{equa3.3})$,
\begin{align*}
g_{k-2,k-1}(x)&=q_{k-1}(x)\\
&~~+ \left(\left[\begin{array}{c|cccc}
            0 & q_0(x) & q_1(x) & \dots & q_{k-2}(x) \\ \hline
            q_{k-1}^{(\alpha_1)}(x_{\beta_1}) & q_0^{(\alpha_1)}(x_{\beta_1}) & q_1^{(\alpha_1)}(x_{\beta_1}) & \dots & q_{k-2}^{(\alpha_1)}(x_{\beta_1}) \\
            q_{k-1}^{(\alpha_2)}(x_{\beta_2}) & q_0^{(\alpha_2)}(x_{\beta_2}) & q_1^{(\alpha_2)}(x_{\beta_2}) & \dots & q_{k-2}^{(\alpha_2)}(x_{\beta_2}) \\
            \vdots & \vdots & \vdots & ~ & \vdots \\
            q_{k-1}^{(\alpha_{k-1})}(x_{\beta_{k-1}}) & q_0^{(\alpha_{k-1})}(x_{\beta_{k-1}}) & q_1^{(\alpha_{k-1})}(x_{\beta_{k-1}}) & \dots & q_{k-2}^{(\alpha_{k-1})}(x_{\beta_{k-1}}) \\
            \end{array}\right]/V_{k-2}\right)\\
         &\triangleq q_{k-1}(x)-\tilde{p}_{k-1}(x),
\end{align*}
and  $\tilde{p}_{k-1}(x)=(q_0(x),q_1(x),\dots,q_{k-2}(x))V_{k-2}^{-1}( q_{k-1}^{(\alpha_1)}(x_{\beta_1}), q_{k-1}^{(\alpha_2)}(x_{\beta_2}),\dots, q_{k-1}^{(\alpha_{k-1})}(x_{\beta_{k-1}}))^\textup{T}$ is precisely the interpolation
polynomial which satisfies the conditions
$$\tilde{p}_{k-1}^{(\alpha_{i})}(x_{\beta_{i}})=q_{k-1}^{(\alpha_i)}(x_{\beta_i}), i=1,2,\dots,k-1.$$
Thus we obtain   $ g_{k-2,k-1}^{(\alpha_i)}(x_{\beta_i})=q_{k-1}^{(\alpha_i)}(x_{\beta_i})-\tilde{p}_{k-1}^{(\alpha_{i})}(x_{\beta_{i}})=0, i=1,2,\dots,k-1.$

Next, we prove the case $i=k$.
The Schur complement of $V_{k-2}$ in $V_{k-1}$ is

\resizebox{\linewidth}{!}{$
\begin{aligned}
    \left(V_{k-1}/V_{k-2}\right) &= \left(\left[
        \begin{array}{cccc|c}
            q_0^{(\alpha_1)}(x_{\beta_1}) & q_1^{(\alpha_1)}(x_{\beta_1}) & \dots & q_{k-2}^{(\alpha_1)}(x_{\beta_1}) & q_{k-1}^{(\alpha_1)}(x_{\beta_1}) \\
            q_0^{(\alpha_2)}(x_{\beta_2}) & q_1^{(\alpha_2)}(x_{\beta_2}) & \dots & q_{k-2}^{(\alpha_2)}(x_{\beta_2}) & q_{k-1}^{(\alpha_2)}(x_{\beta_2}) \\
            \vdots & \vdots & ~ & \vdots & \vdots \\
            q_0^{(\alpha_{k-1})}(x_{\beta_{k-1}}) & q_1^{(\alpha_{k-1})}(x_{\beta_{k-1}}) & \dots & q_{k-2}^{(\alpha_{k-1})}(x_{\beta_{k-1}}) & q_{k-1}^{(\alpha_{k-1})}(x_{\beta_{k-1}}) \\ \hline
            q_0^{(\alpha_{k})}(x_{\beta_{k}}) & q_1^{(\alpha_{k})}(x_{\beta_{k}}) & \dots & q_{k-2}^{(\alpha_{k})}(x_{\beta_{k}}) & q_{k-1}^{(\alpha_{k})}(x_{\beta_{k}}) \\
        \end{array}
    \right]/V_{k-2}\right) \nonumber \\
    &= \left(\left[
        \begin{array}{c|cccc}
            q_{k-1}^{(\alpha_{k})}(x_{\beta_{k}}) & q_0^{(\alpha_{k})}(x_{\beta_{k}}) & q_1^{(\alpha_{k})}(x_{\beta_{k}}) & \dots & q_{k-2}^{(\alpha_{k})}(x_{\beta_{k}}) \\ \hline
             q_{k-1}^{(\alpha_1)}(x_{\beta_1}) & q_0^{(\alpha_1)}(x_{\beta_1}) & q_1^{(\alpha_1)}(x_{\beta_1}) & \dots & q_{k-2}^{(\alpha_1)}(x_{\beta_1}) \\
            q_{k-1}^{(\alpha_2)}(x_{\beta_2}) &  q_0^{(\alpha_2)}(x_{\beta_2}) & q_1^{(\alpha_2)}(x_{\beta_2}) & \dots & q_{k-2}^{(\alpha_2)}(x_{\beta_2}) \\
            \vdots & \vdots & \vdots & ~ & \vdots \\
            q_{k-1}^{(\alpha_{k-1})}(x_{\beta_{k-1}}) &  q_0^{(\alpha_{k-1})}(x_{\beta_{k-1}}) & q_1^{(\alpha_{k-1})}(x_{\beta_{k-1}}) & \dots & q_{k-2}^{(\alpha_{k-1})}(x_{\beta_{k-1}})
        \end{array}\right]/V_{k-2}\right) \nonumber\\
    &=  \delta_{x_{\beta_k}} \circ D_x^{\alpha_k}\left(\left(
        \begin{array}{c|cccc}
            q_{k-1}(x) & q_0(x) & q_1(x) & \dots & q_{k-2}(x) \\ \hline
           q_{k-1}^{(\alpha_1)}(x_{\beta_1}) & q_0^{(\alpha_1)}(x_{\beta_1}) & q_1^{(\alpha_1)}(x_{\beta_1}) & \dots & q_{k-2}^{(\alpha_1)}(x_{\beta_1}) \\
            q_{k-1}^{(\alpha_2)}(x_{\beta_2}) &  q_0^{(\alpha_2)}(x_{\beta_2}) & q_1^{(\alpha_2)}(x_{\beta_2}) & \dots & q_{k-2}^{(\alpha_2)}(x_{\beta_2}) \\
            \vdots & \vdots & \vdots & ~ & \vdots \\
            q_{k-1}^{(\alpha_{k-1})}(x_{\beta_{k-1}}) &  q_0^{(\alpha_{k-1})}(x_{\beta_{k-1}}) & q_1^{(\alpha_{k-1})}(x_{\beta_{k-1}}) & \dots & q_{k-2}^{(\alpha_{k-1})}(x_{\beta_{k-1}})
        \end{array}\right)/V_{k-2}\right) \nonumber \\
    &= \delta_{x_{\beta_k}} \circ D_x^{\alpha_k}(g_{k-2,k-1}(x))\nonumber \\
    & = g_{k-2,k-1}^{(\alpha_k)}(x_{\beta_k}).\nonumber
\end{aligned}
$}
The second equality follows directly from Proposition \ref{prop:change}, and the third equality is due to Proposition \ref{prop:E}. From  Proposition \ref{prop:rank} we know that
\begin{displaymath}
g_{k-2,k-1}^{(\alpha_k)}(x_{\beta_k})=(V_{k-1}/V_{k-2})=|(V_{k-1}/V_{k-2})|=\frac{|V_{k-1}|}{|V_{k-2}|}.
\end{displaymath}
Since $\{q_0(x), q_1(x),\dots,q_{N-1}(x)\}$ is a strongly proper interpolation basis, it  follows that $|V_{k-1}|\neq 0, |V_{k-2}|\neq 0$.  This completes the proof of the proposition.
\end{proof}

\begin{theorem}\label{TH7}
Let $\textup{span}\{q_0(x), q_1(x),\dots,q_{N-1}(x)\}$ be a strongly proper interpolation space for the Birkhoff interpolation problem $(\ref{equa2.1})$.
$g_{j,k-1}(x)$, $k=1,2,\dots,N, j=0,1,\dots,k-2$ are defined by $(\ref{equa3.3})$.
Let $p_{-1}(x)=0$, then the interpolation polynomial 
\begin{equation*}
p_{k-1}(x)=p_{k-2}(x)+\frac{y_{\beta_k,\alpha_k}-p_{k-2}^{(\alpha_k)}(x_{\beta_k})}{g_{k-2,k-1}^{(\alpha_k)}(x_{\beta_k})}g_{k-2,k-1}(x),~~k=1, 2, \dots, N.
\end{equation*}
\end{theorem}

\begin{proof}
First we will write

\resizebox{\linewidth}{!}{$
\begin{aligned}
p_{k-1}(x)&=-\left(\left[\begin{array}{c|cccc|c}
           0 & q_0(x) & q_1(x) & \dots & q_{k-2}(x) & q_{k-1}(x) \\ \hline
           y_{\beta_1,\alpha_1} & q_0^{(\alpha_1)}(x_{\beta_1}) & q_1^{(\alpha_1)}(x_{\beta_1}) & \dots & q_{k-2}^{(\alpha_1)}(x_{\beta_1}) & q_{k-1}^{(\alpha_1)}(x_{\beta_1}) \\
           y_{\beta_2,\alpha_2} & q_0^{(\alpha_2)}(x_{\beta_2}) & q_1^{(\alpha_2)}(x_{\beta_2}) & \dots & q_{k-2}^{(\alpha_2)}(x_{\beta_2}) & q_{k-1}^{(\alpha_2)}(x_{\beta_2}) \\
           \vdots & \vdots & \vdots & ~ & \vdots & \vdots \\
           y_{\beta_{k-1},\alpha_{k-1}} & q_0^{(\alpha_{k-1})}(x_{\beta_{k-1}}) & q_1^{(\alpha_{k-1})}(x_{\beta_{k-1}}) & \dots & q_{k-2}^{(\alpha_{k-1})}(x_{\beta_{k-1}}) & q_{k-1}^{(\alpha_{k-1})}(x_{\beta_{k-1}}) \\ \hline
           y_{\beta_k,\alpha_k} & q_0^{(\alpha_k)} (x_{\beta_k})& q_1^{(\alpha_k)}(x_{\beta_k}) & \dots & q_{k-2}^{(\alpha_k)}(x_{\beta_k}) & q_{k-1}^{(\alpha_k)}(x_{\beta_k}) \\
         \end{array}\right]/V_{k-1}\right)\\
         &\triangleq-\left(\left[\begin{array}{ccc}
0 & Q_{k-2}(x) & q_{k-1}(x)\\
Y_{k-1} & V_{k-2} & B\\
y_{\beta_{k},\alpha_{k}} &  Q_{k-2}^{\alpha_{k}}(x_{\beta_{k}}) & q_{k-1}^{(\alpha_k)}(x_{\beta_k})\\
\end{array}\right]/\left(
                     \begin{array}{cc}
                        V_{k-2}  & B \\
                        Q_{k-2}^{\alpha_{k}}(x_{\beta_{k}}) & q_{k-1}^{(\alpha_k)}(x_{\beta_k}) \\
                     \end{array}
                   \right)
\right).
\end{aligned}
$}
Then  Proposition \ref{prop:identity} shows that $p_{k-1}(x)$ can be written as
\begin{align*}
p_{k-1}(x)&=-\left(\left(
                       \begin{array}{cc}
                         0 & Q_{k-2}(x) \\
                         Y_{k-1} & V_{k-2} \\
                       \end{array}
                     \right)
/V_{k-2}\right)+\left(\left(
       \begin{array}{cc}
         Q_{k-2}(x) & q_{k-1}(x) \\
         V_{k-2} & B \\
       \end{array}
     \right)
/V_{k-2}\right)\\
&~~~~~~~~\left(\left(
            \begin{array}{cc}
              V_{k-2} & B \\
              Q_{k-2}^{\alpha_{k}}(x_{\beta_{k}}) &  q_{k-1}^{(\alpha_k)}(x_{\beta_k})\\
            \end{array}
          \right)
/V_{k-2}\right)^{-1}
\left(\left(
                \begin{array}{cc}
                  Y_{k-1} & V_{k-2} \\
                  y_{\beta_{k},\alpha_{k}} & Q_{k-2}^{\alpha_{k}}(x_{\beta_{k}})  \\
                \end{array}
              \right)
/V_{k-2}\right).
\end{align*}
On the other hand, it is easy to see that
\begin{align*}
-\left(\left(
                       \begin{array}{cc}
                         0 & Q_{k-2}(x) \\
                         Y_{k-1} & V_{k-2} \\
                       \end{array}
                     \right)
/V_{k-2}\right)= p_{k-2}(x);
\end{align*}
by Proposition \ref{prop:change} and Proposition \ref{prop:V.V} we know that 
\begin{align*}
&\left(\left(
       \begin{array}{cc}
         Q_{k-2}(x) & q_{k-1}(x) \\
         V_{k-2} & B \\
       \end{array}
     \right)
/V_{k-2}\right)
=\left(\left(
       \begin{array}{cc}
         q_{k-1}(x) & Q_{k-2}(x) \\
         B & V_{k-2} \\
       \end{array}
     \right)
/V_{k-2}\right)
=g_{k-2,k-1}(x);\\
&   \left(\left(
            \begin{array}{cc}
              V_{k-2} & B \\
              Q_{k-2}^{\alpha_{k}}(x_{\beta_{k}}) & q_{k-1}^{(\alpha_k)}(x_{\beta_k}) \\
            \end{array}
          \right)
/V_{k-2}\right)^{-1}
= \left(V_{k-1}/V_{k - 2}\right)^{-1}
=\dfrac{1}{g_{k-2,k-1}^{(\alpha_k)}(x_{\beta_k})}.
\end{align*}
Applying Propositon \ref{prop:change} and Proposition \ref{prop:E} , we can obtain
\begin{align*}
    \left(\left(
                \begin{array}{cc}
                  Y_{k-1} & V_{k-2} \\
                  y_{\beta_{k},\alpha_{k}} & Q_{k-2}^{\alpha_{k}}(x_{\beta_{k}}) \\
                \end{array}
              \right)
/V_{k-2}\right)
&=\left(\left(
                \begin{array}{cc}
                  y_{\beta_{k},\alpha_{k}}& Q_{k-2}^{\alpha_{k}}(x_{\beta_{k}}) \\
                  Y_{k-1} & V_{k-2} \\
                \end{array}
              \right)
/V_{k-2}\right)\\
&=y_{\beta_{k},\alpha_{k}}-Q_{k-2}^{\alpha_{k}}(x_{\beta_{k}})V_{k-2}^{-1}Y_{k-1}\\
&= y_{\beta_k,\alpha_k} - p_{k-2}^{(\alpha_k)}(x_{\beta_k}).
\end{align*}

Substituting the above equations into the expression for $p_{k-1}(x)$ yields the conclusion of the theorem.
\end{proof}

\begin{theorem}\label{gx} 
 Under the same assumptions as Theorem $\ref{TH7}$, for any fixed $k=1,2,\dots,N$,
\begin{equation*}
g_{j,k-1}(x)=g_{j-1,k-1}(x)-\frac{g_{j-1,k-1}^{(\alpha_{j+1})}(x_{\beta_{j+1}})}{g_{j-1,j}^{(\alpha_{j+1})}(x_{\beta_{j+1}})}g_{j-1,j}(x),~ j=0, 1, \dots, k-2.
\end{equation*}
\end{theorem}

\begin{proof}
For any fixed $k=1,2,\dots,N$, note that $g_{j,k-1}(x), j=0,1,\dots,k-2$ can be written as  follows:

\resizebox{\linewidth}{!}{$
\begin{aligned}
g_{j,k-1}(x)
&=\left(\left[\begin{array}{c|cccc|c}
         q_{k-1}(x) & q_0(x) & q_1(x) & \dots & q_{j-1}(x) & q_{j}(x) \\ \hline
        q_{k-1}^{(\alpha_1)}(x_{\beta_1}) & q_0^{(\alpha_1)}(x_{\beta_1}) & q_1^{(\alpha_1)}(x_{\beta_1}) & \dots & q_{j-1}^{(\alpha_1)}(x_{\beta_1}) & q_{j}^{(\alpha_1)}(x_{\beta_1}) \\
        q_{k-1}^{(\alpha_2)}(x_{\beta_2}) & q_0^{(\alpha_2)}(x_{\beta_2}) & q_1^{(\alpha_2)}(x_{\beta_2}) & \dots & q_{j-1}^{(\alpha_2)}(x_{\beta_2}) & q_{j}^{(\alpha_2)}(x_{\beta_2}) \\
        \vdots & \vdots & \vdots & ~ & \vdots & \vdots \\
        q_{k-1}^{(\alpha_j)}(x_{\beta_j}) & q_0^{(\alpha_j)}(x_{\beta_j}) & q_1^{(\alpha_j)}(x_{\beta_j}) & \dots & q_{j-1}^{(\alpha_j)}(x_{\beta_j}) & q_{j}^{(\alpha_j)}(x_{\beta_j}) \\ \hline
        q_{k-1}^{(\alpha_{j+1})}(x_{\beta_{j+1}}) & q_0^{(\alpha_{j+1})}(x_{\beta_{j+1}}) & q_1^{(\alpha_{j+1})}(x_{\beta_{j+1}}) & \dots & q_{j-1}^{(\alpha_{j+1})}(x_{\beta_{j+1}}) & q_{j}^{(\alpha_{j+1})}(x_{\beta_{j+1}})\\
        \end{array}\right]/V_{j}\right)\\
&\triangleq-\left(\left[\begin{array}{ccc}
q_{k-1}(x) & Q_{j-1}(x) & q_{j}(x)\\
\tilde{B_{1}} & V_{j-1} & \tilde{B}\\
q_{k-1}^{(\alpha_{j+1})}(x_{\beta_{j+1}}) &  Q_{j-1}^{\alpha_{j+1}}(x_{\beta_{j+1}}) & q_{j}^{(\alpha_j+1)}(x_{\beta_{j+1}})\\
\end{array}\right]/\left(
                     \begin{array}{cc}
                         V_{j-1} & \tilde{B} \\
                       Q_{j-1}^{\alpha_{j+1}}(x_{\beta_{j+1}}) & q_{j}^{(\alpha_j+1)}(x_{\beta_{j+1}})\\
                     \end{array}
                   \right)
\right).
\end{aligned}
$}
The conclusion follows analogously to the proof of the previous theorem.
\end{proof}

\section{Algorithms for computing Newton-type basis and interpolation polynomial}\label{sec4}

\subsection{Algorithm 1}
Under the conditions of Proposition \ref{prop:V.V} and Theorem \ref{TH7}, we know that $\{g_{k-2,k-1}(x),k=1,2,\dots,N\}$
defined by equation $(\ref{equa3.3})$ is the desired Newton-type basis, and the polynomial $p_{N-1}(x)$ defined by equation $(\ref{equa3.2})$ is the  interpolation polynomial  for the interpolation conditions specified in (\ref{equa2.1}). 
Therefore, the central task now is to determine a set of strongly proper basis $\{q_{0}(x),q_{1}(x),\dots,q_{N-1}(x)\}$. The algorithms below incorporate a core judgment criterion at each iterative step (see line 18 of Algorithm 1 and line 17 of Algorithm 2) to ensure this condition holds.
In the following algorithm, each $g_{-1,i}(x),i=0,1,\dots,N-1$ is updated at every iteration if it fails to  satisfy the judgment criterion, and the final output set 
$\{g_{-1,i}(x),i=0,1,\dots,N-1 \}$ corresponds precisely to the strongly proper basis $\{q_0(x),q_1(x),\dots,q_{N-1}(x)\}$.

\begin{algorithm}
	\caption{~}
	\label{alg:revredbasis}
	\begin{algorithmic}[1]
    \State{\textbf{Input}}:
    \State{The Birkhoff interpolation conditions   $\delta_{x_{\beta_1}}\circ D_x^{\alpha_1},\delta_{x_{\beta_2}}\circ D_x^{\alpha_2},\dots,\delta_{x_{\beta_N}}\circ D_x^{\alpha_N},$\\
    where $T = [(\beta_1, \alpha_1), (\beta_2, \alpha_2), \dots, (\beta_N, \alpha_N)]$ is an  $N-$DOS.\\ 
    The interpolation values $y_{\beta_1,\alpha_1}, y_{\beta_2,\alpha_2},\dots, y_{\beta_N,\alpha_N}$}.
    \State{\textbf{Output}}: 
    \State{The monomial basis $\{g_{-1,0}(x),g_{-1,1}(x),\dots,g_{-1,N-1}(x) \}.$\\The Newton-type basis $\{g_{-1,0}(x),g_{0,1}(x),\dots,g_{N-2,N-1}(x) \}$. \\ The interpolation polynomial $p_{N-1}(x)$}.
    \State{ $//$ Initialization:\\
     $g_{-1,i}(x):=x^{\alpha_1+i}, i=0,1, \dots, N-1$,\\
    $p_{-1}(x):=0$,
    $p_{0}(x):=p_{-1}(x)+\dfrac {y_{\beta_1,\alpha_1}-p_{-1}^{(\alpha_1)}(x_{\beta_1})}{g_{-1,0}^{(\alpha_1)}(x_{\beta_1})}g_{-1,0}(x)=\frac{y_{\beta_1,\alpha_1}}{\alpha _1!}x^{\alpha _1}$,
    
    }
    \State{List:=$\{2,3,\dots,N\}$}.
    \State{$\rm{\textbf{while\ }} List \neq \varnothing \rm{\textbf{\ do}}$}
       \State{\hspace{1cm}$k$:=Min(List)};
       \State{\hspace{1cm}\rm{\textbf{for}\ $j=0,1,\dots,k-2$}}
          \State{$\hspace{2cm}g_{j,k-1}(x):=g_{j-1,k-1}(x)-\dfrac{g_{j-1,k-1}^{(\alpha_{j+1})}(x_{\beta_{j+1}})}{g_{j-1,j}^{(\alpha_{j+1})}(x_{\beta_{j+1}})}g_{j-1,j}(x)$}
       \State{$\hspace{1cm}{\rm{\textbf{end}}}\hspace{2mm} j$}
       \State{$\hspace{1cm}{\rm{\textbf{if}}}\ g_{k-2,k-1}^{(\alpha_k)}(x_{\beta_k})\neq 0$}
          \State{\hspace{2cm}{List:=List}$-\{k\}$}
    \State{\hspace{2cm}$p_{k-1}(x):=p_{k-2}(x)+\dfrac {y_{\beta_k,\alpha_k}-p_{k-2}^{(\alpha_k)}(x_{\beta_k})}{g_{k-2,k-1}^{(\alpha_k)}(x_{\beta_k})}g_{k-2,k-1}(x)$}
       \State{\hspace{1cm}{\rm{\textbf{else}}}}
          \State{$\hspace{2cm}{\rm{\textbf{for}}}\ d=k-1, k,\dots,N-1$}
             \State{\hspace{3cm}$g_{-1,d}(x):=x*g_{-1,d}(x)$}
          \State{$\hspace{2cm}{\rm{\textbf{end}}}\hspace{2mm} d $}
       \State{\hspace{1cm}{\rm{\textbf{end\ if}}}}
    \State{\rm{\textbf{end\ while}}}
	\end{algorithmic}
\end{algorithm}
\newpage
From Theorem $\ref{zhongzhi}$ we know   that the proposed  algorithm terminates. The following result shows its correctness.
\begin{theorem}
 The output  $\{g_{-1,0}(x), g_{0,1}(x), \dots, g_{N-2,N-1}(x)\}$ is a Newton-type   basis for the Birkhoff interpolation problem $(\ref{equa2.1})$. Furthermore, it is  a set of strongly  proper interpolation basis, and $p_{N-1}(x)$ is the corresponding interpolation polynomial.
\end{theorem}
\begin{proof}
From Proposition \ref{prop:V.V} and Theorem \ref{gx}, to prove that
\begin{displaymath}
  \{g_{-1,0}(x), g_{0,1}(x), \dots, g_{N-2,N-1}(x)\}  
\end{displaymath}
is a Newton-type basis,
we only need to prove that the recursively updated
\begin{displaymath}
  \{g_{-1,0}(x), g_{-1,1}(x), \dots, g_{-1,N-1}(x)\}  
\end{displaymath}
 is a strongly proper basis.

We will use induction on the number of interpolation conditions.
When $k=1$, $ \delta_{x_{\beta_1}}\circ D_x^{\alpha_1}(g_{-1,0}(x))= \delta_{x_{\beta_1}}\circ D_x^{\alpha_1}(x^{\alpha_1})=\alpha_1!\neq 0.$
When $k=2$, $g_{-1,1}(x)=x^{\alpha_1+1+s}, s=1,2,\dots$. This corresponds to the execution process of $k=2$ in Algorithm 1. 

First, assume $g_{-1,1}(x)=x^{\alpha_1+1}$ holds and compute $g_{0,1}(x)$ using line 16 of Algorithm 1. If $g_{0,1}(x)$ satisfies the judgment condition (line 18 in Algorithm 1), this computational step is completed and
$g_{-1,1}(x),g_{0,1}(x)$ are stored. Otherwise, the degree of $g_{-1,1}(x)$ is increased, and 
$g_{0,1}(x)$ is recomputed recursively until a 
$g_{0,1}(x)$
 that satisfies 
$g_{0,1}^{(\alpha_2)}(x_{\beta_2})\neq 0$
 is obtained.
It is straightforward to verify that 
\begin{displaymath}
    g_{0,1}^{(\alpha_1)}(x_{\beta_1})=g_{-1,1}^{(\alpha_1)}(x_{\beta_1})-\frac{g_{-1,1}^{(\alpha_1)}(x_{\beta_1})}{g_{-1,0}^{(\alpha_1)}(x_{\beta_1})}g_{-1,0}^{(\alpha_1)}(x_{\beta_1})=0.
\end{displaymath} 
Thus, $\{g_{-1,0}(x),g_{0,1}(x)\}$ forms a set of strongly proper Newton-type basis for the first two interpolation conditions. Since $g_{0,1}(x)$ is a linear combination of $g_{-1,1}(x)$ and $g_{-1,0}(x)$, 
$\{g_{-1,0}(x),g_{-1,1}(x)\}$ constitutes a set of strongly proper monomial basis for the first two conditions.

Now assume the result has been proved for the first  $k-1$ interpolation conditions, that is,
$\{g_{-1,0}(x), g_{-1,1}(x), \dots, g_{-1,k-2}(x)\}$ is a strongly proper basis. By Proposition \ref{prop:V.V}, we know that $\{g_{-1,0}(x), g_{0,1}(x), \dots, g_{k-3,k-2}(x)\}$ is a Newton-type basis. 
First, we prove that $\{g_{-1,0}(x), g_{0,1}(x), \dots, g_{k-2,k-1}(x)\}$ is a Newton-type basis for the first $k$ conditions. By the induction hypothesis, it suffices to further verify that 
\begin{displaymath}
  g_{k-2,k-1}(x)=g_{k-3,k-1}(x)-\frac{g_{k-3,k-1}^{(\alpha_{k-1})}(x_{\beta_{k-1}})}{g_{k-3,k-2}^{(\alpha_{k-1})}(x_{\beta_{k-1}})}g_{k-3,k-2}(x) , 
\end{displaymath}
 satisfies: 
$\textup{(\romannumeral 1) }g_{k-2,k-1}^{(\alpha_k)}(x_{\beta_k})\neq 0;$ 
$\textup{(\romannumeral 2) }g_{k-2,k-1}^{(\alpha_i)}(x_{\beta_i})=0, \forall i=1,2,\dots,k-1$.

By line 18 of Algorithm 1, we know that the first inequality holds.
Note that $g_{-1,k-1}(x), g_{0,k-1}(x), \dots, g_{k-2,k-1}(x)$ are polynomials of the same degree. 
For $i=0,1,\dots,k-2$, it is easy to verify that $g_{i,k-1}(x)$ vanishes at the first $i+1$ conditions, i.e.,
\begin{displaymath}
    g_{i,k-1}^{(\alpha_j)}(x_{\beta_j})=0, \forall j=1,2,\dots,i+1.
\end{displaymath}
Furthermore, 
\begin{displaymath}
    g_{k-2,k-1}^{(\alpha_i)}(x_{\beta_i})=g_{k-3,k-1}^{(\alpha_i)}(x_{\beta_i})-\frac{g_{k-3,k-1}^{(\alpha_{k-1})}(x_{\beta_{k-1}})}{g_{k-3,k-2}^{(\alpha_{k-1})}(x_{\beta_{k-1}})}g_{k-3,k-2}^{(\alpha_i)}(x_{\beta_i}),  \forall i=1,2,\dots,k-1.
\end{displaymath}
It is obvious that $g_{k-2,k-1}^{(\alpha_i)}(x_{\beta_i})=0$ when $i=k-1$. For $i=1,2,\dots,k-2$, as shown in the preceding proof, we know that 
$g_{k-3,k-1}^{(\alpha_i)}(x_{\beta_i})=0$, and by the induction hypothesis, \[
\left\{
\begin{aligned}
&g_{k-3,k-2}^{(\alpha_i)}(x_{\beta_i})=0, \\
&g_{k-3,k-2}^{(\alpha_{k-1})}(x_{\beta_{k-1}})\neq 0.
\end{aligned}
\right.
\] 
Hence
\begin{displaymath}
    g_{k-2,k-1}^{(\alpha_i)}(x_{\beta_i})=0, \forall i=1,2,\dots,k-1.
\end{displaymath} Thus we have proved that $\{g_{-1,0}(x), g_{0,1}(x), \dots, g_{k-2,k-1}(x)\}$ forms a set of strongly proper Newton-type basis.

Since $$\textup{span}\{g_{-1,0}(x), g_{-1,1}(x), \dots, g_{-1,k-1}(x)\}=\textup{span}\{g_{-1,0}(x),g_{0,1}(x),\dots,g_{k-2,k-1}(x)\},$$ 
it follows that
$\{g_{-1,0}(x), g_{-1,1}(x), \dots, g_{-1,k-1}(x)\}$ constitutes a set of strongly proper monomial basis.
\end{proof}

For a given Birkhoff interpolation problem $(\ref{equa2.1})$  with $N$ interpolation conditions, Algorithm 1 computes $N(N-1)/2$ auxiliary polynomials in the order: $g_{-1,0}(x) \rightarrow g_{-1,1} (x)\rightarrow g_{0,1}(x) \rightarrow g_{-1,2}(x) \rightarrow g_{0,2}(x)\rightarrow g_{1,2}(x)  \rightarrow \dots$. That is,  if these auxiliary polynomials are arranged in a matrix form according to their subscripts (the row index starts from -1, the column index starts from 0), the algorithm computes them column by column during execution. From the perspective of the Vandermonde matrix corresponding to the interpolation problem, it is easy to verify the following conclusion.
\begin{theorem}\label{theo10}
For any fixed $k=2,3,\dots,N$, suppose that  the computation of the first $k-2$ columns has already been completed. Then,
for $j=0,1,\dots,k-2$,  the combination coefficients of $g_{j,k-1}(x)$ are  precisely the elimination coefficients obtained when performing column elimination on the Vandermonde matrix constructed from the basis
$$\{g_{-1,0}(x), g_{0,1}(x), \dots, g_{k-3,k-2}(x), g_{j-1,k-1}(x)\}$$
under the first $k$ interpolation conditions—specifically, elimination is carried out using the nonzero diagonal element in the $j$-th column to annihilate $g_{j-1,k-1}^{(\alpha_{j+1})}(x_{\beta{j+1}})$ to zero.
\end{theorem}
\noindent\textbf{Example 1} 
We consider the same example as in \cite{wen27}.
Let $x_0=1, x_1=2, x_2=3$, $y_{0,0}=5, y_{1,1}=6, y_{1,2}=4, y_{2,2}=7$.
The interpolation conditions are sorted into $N-$DOS as follows
\begin{displaymath}
(\beta_1,\alpha_1)=(0,0),\ (\beta_2,\alpha_2)=(1,1),\ (\beta_3,\alpha_3)=(1,2),\ (\beta_4,\alpha_4)=(2,2).
\end{displaymath}
According to Algorithm 1, all  auxiliary polynomials are listed in Table \ref{tbl1}.

\begin{table}[ht] \label{tbl1}
\centering
\caption{The auxiliary polynomials.}
\label{tbl1}
\begin{tabular}{c c c c}
\toprule
\(g_{-1,0}(x)=1\) & \(g_{-1,1}(x)=x\) & \(g_{-1,2}(x)=x^2\) & \(g_{-1,3}(x)=x^3\) \\
\midrule
 & \(g_{0,1}(x)=x-1\) & \(g_{0,2}(x)=x^2-1\) & \(g_{0,3}(x)=x^3-1\) \\
 &  & \(g_{1,2}(x)=x^2-4x+3\) & \(g_{1,3}(x)=x^3-12x+11\) \\
 &  &  & \(g_{2,3}(x)=x^3-6x^2+12x-7\) \\
\bottomrule
\end{tabular}
\end{table}

We can verify that the Vandermonde matrix takes the following form
\[
\begin{array}{c@{\hspace{8pt}}c}
  & \begin{array}{ccccc}
      \text{} & g_{-1,0}(x) & g_{0,1}(x) & g_{1,2}(x) & g_{2,3}(x) \\[4pt]
    \end{array} \\[2pt]
  \begin{array}{c}
  \renewcommand{\arraystretch}{2.5}
    \delta_{x_0} \\[8pt]
    \delta_{x_1}\circ D_x \\[8pt]
    \delta_{x_1}\circ D_x^2 \\[8pt]
     \delta_{x_2}\circ D_x^2  \\[8pt]
  \end{array}
  &  
  \left(
  \renewcommand{\arraycolsep}{12pt}
  \renewcommand{\arraystretch}{1.0} 
    \begin{array}{cccc}
     ~~ 1      &~~ 0      & ~~~~0    & ~~~0 \\[8pt]  
      ~~*      & ~~2      & ~~~~0    &~~~ 0 \\[8pt]  
      ~~*      & ~~*      & ~~~~-1    &~~~ 0 \\[8pt]  
      ~~*      & ~~*      & ~~~~*    &~~~ 4 \\[8pt]  
    \end{array}
  \right).
\end{array}
\]
The elements in the strictly lower triangular part of the matrix need not be calculated. $\{1, x-1, x^2-4x+3, x^3-6x^2+12x-7\}$ is a Newton-type basis and the interpolation polynomial is
\begin{displaymath}
p_3(x)=\frac{1}{2}x^3-x^2+4x+\frac{3}{2}.
\end{displaymath}
The set $\{1, x, x^2, x^3\}$ forms a minimal-degree basis, and consequently, $p_3(x)$ is the minimal-degree interpolation polynomial.
Applying the algorithm from \cite{wen27} to this example yields an interpolation polynomial of degree 6.
\begin{remark}
Theorem \ref{theo10} indicates that the construction of 
 the auxiliary polynomials in Algorithm 1 implicitly embodies the idea of Gaussian elimination. Compared with existing works that apply Gaussian elimination to solve Birkhoff interpolation problems \cite{lei}, the present approach differs mainly in the following two aspects: 1) During the iterative procedure of Algorithm 1, when computing the $k$-th Newton interpolation basis, $k=1,2,\dots,N$, only the first $k$ interpolation conditions are evaluated on the auxiliary polynomial and processed by Gaussian elimination, without considering all $N$ conditions. As illustrated in Example 1, the entries below the main diagonal of the resulting Vandermonde matrix are never accessed during the algorithm, which reduces computational cost.
 2) In the process of sequentially computing all auxiliary polynomials in the $(k-1)$-th column, once a later polynomial is computed using the previous one, the earlier polynomials need not be retained. Ultimately, only the diagonal term $g_{k-2,k-1}(x)$ needs to be preserved. This  significantly reduces storage space.
\end{remark}

\begin{remark}
While both the proposed algorithm and the method of \cite{wen27} are generalizations of GRPIA \cite{wen20} and employ the Schur complement theory within a univariate Birkhoff interpolation framework, they are rooted in fundamentally different principles. In \cite{wen27},  the authors address the Birkhoff interpolation problem by constructing high-degree interpolation bases that admit factorization, ultimately yielding a simplified version of the RHBPIA algorithm. The proposed Algorithm 1 introduces a criterion to ensure the well-posedness of the Birkhoff interpolation problem, thereby yielding lower-degree Newton bases and interpolation polynomials. We further propose Algorithm 2 to address a more general Birkhoff interpolation problem whose conditions are composed of evaluation functionals and differential polynomials, extending beyond those involving only partial derivatives. This extension represents another key difference from the RHBPIA method.
\end{remark}

\noindent\textbf{Example 2}
Let $x_0=-1, x_1=0, x_2=1$, $y_{0,0}=2, y_{1,1}=4, y_{2,0}=6, y_{2,1}=8$.
By Definition \ref{def:order},
\begin{displaymath}
(\beta_1,\alpha_1)=(0,0),\ (\beta_2,\alpha_2)=(2,0),\ (\beta_3,\alpha_3)=(1,1),\ (\beta_4,\alpha_4)=(2,1).
\end{displaymath}
According to Algorithm 1, the auxiliary polynomials are shown in the following table.
\begin{table}[ht]
\centering
\caption{The auxiliary polynomials.}\label{tbl2}
\begin{tabular}{c c c c}
\toprule
\(g_{-1,0}(x)=1\) & \(g_{-1,1}(x)=x\) & \(g_{-1,2}(x)=x^3\) & \(g_{-1,3}(x)=x^4\) \\
\midrule
 & \(g_{0,1}(x)=x+1\) & \(g_{0,2}(x)=x^3+1\) & \(g_{0,3}(x)=x^4-1\) \\
 &  & \(g_{1,2}(x)=x^3-x\) & \(g_{1,3}(x)=x^4-1\) \\
 &  &  & \(g_{2,3}(x)=x^4-1\) \\
\bottomrule
\end{tabular}
\end{table}

$\{g_{-1,0}(x),g_{-1,1}(x),g_{-1,2}(x),g_{-1,3}(x)\}=\{1,x,x^3,x^4\}$ is a strongly proper monomial basis. It is easy to verify that
\begin{displaymath}
g_{-1,0}(x_0)=1,\  g_{0,1}(x_2)=2,\   g_{1,2}^{(1)}(x_1)=-1,\  g_{2,3}^{(1)}(x_2)=4.
\end{displaymath}

  
The set $\{1, x+1, x^3 - x, x^4 - 1\}$ constitutes the Newton-type basis with respect to the given interpolation conditions. The interpolation polynomial computed by the recursive algorithm is
\begin{displaymath}
p_3(x)=\frac{5}{2}x^4-2x^3+4x+\frac{3}{2}.
\end{displaymath}

\subsection{Algorithm 2}

By computing the Vandermonde determinant, we can verify that the set $\{1, x, x^2, x^3\}$ also constitutes a proper monomial basis for Example 2.
This implies that the interpolation polynomial produced by Algorithm 1 is not necessarily of minimal degree.
Furthermore, Algorithm 1 is restricted to cases where the interpolation conditions consist of monomial differential operators and evaluation functionals.
We enhance Algorithm 1, making it capable of handling interpolation conditions formed by the composition of differential polynomials and evaluation functionals.
Moreover, Algorithm 2 offers the additional benefit of degree reduction in specific examples by reordering the interpolation conditions.

In Algorithm 2, we let $L_1,L_2,\dots,L_N$  denote the interpolation conditions, where every $L_i, i=1,2,\dots,N$ is the composite operation of an evaluation functional and a differential polynomial, and we sort them in ascending order of the highest derivative order they involve.  
Namely, if we let $\alpha_i$ be the highest-order derivative of $L_i$, then $\alpha_1\leq \alpha_2\leq\dots\leq \alpha_N$.
For example, suppose the interpolation conditions at $x_0$ are $\delta_{x_0}\circ D_x, \delta_{x_0}\circ (1+D_x^2)$, and the interpolation conditions at $x_1$ are $\delta_{x_1}\circ (1+D_x), \delta_{x_1}\circ (D_x^2+D_x^3)$, we then sort the interpolation conditions into the sequence:
\begin{displaymath}
L_1=\delta_{x_0}\circ D_x,\ 
L_2=\delta_{x_1}\circ (1+D_x),\ L_3=\delta_{x_0}\circ (1+D_x^2),\ L_4=\delta_{x_1}\circ (D_x^2+D_x^3). 
\end{displaymath}

\begin{algorithm}
	\caption{~}
	\label{alg:revredbasis}
	\begin{algorithmic}[1]
    \State{\textbf{Input}}:
    \State{The Birkhoff interpolation conditions $L_1, L_2, \dots, L_N.$\\
    The highest-order derivatives $\alpha_1,\alpha_2,\dots,\alpha_N 
\ (\alpha_1\leq \alpha_2\leq\dots\leq \alpha_N)$.\\
    The interpolation values $y_1, y_2,\dots,y_N$}.
    \State{\textbf{Output}}: 
    \State{The Newton-type basis $\{g_{-1,0}(x),g_{0,1}(x),\dots,g_{N-2,N-1}(x) \}$. \\ The interpolation polynomial $p_{N-1}(x)$.}
    \State{ $//$ Initialization:\\
     $g_{-1,i}(x):=x^{\alpha_1+i}, i=0,1, \dots, N-1$,\\
    $p_{-1}(x):=0$,
    $p_{0}(x):=p_{-1}(x)+\dfrac {y_1-L_1(p_{-1}(x))}{L_1(g_{-1,0}(x))}g_{-1,0}(x)=\frac{y_1}{\alpha _1!}x^{\alpha _1}$,
    \State{List:=$\{2,3,\dots,N\}$}.
    \State{$\rm{\textbf{while\ }} List \neq  \varnothing\rm{\textbf{\ do}}$}
       \State{\hspace{1cm}$k:$=Min(List)};
       \State{\hspace{1cm}\rm{\textbf{for}\ $j=0,1,\dots,k-2$}}
          \State{$\hspace{2cm}g_{j,k-1}(x):=g_{j-1,k-1}(x)-\dfrac{L_{j+1}(g_{j-1,k-1}(x))}{L_{j+1}(g_{j-1,j}(x))}g_{j-1,j}(x)$}
       \State{$\hspace{1cm}{\rm{\textbf{end}}}\hspace{2mm} j$}
       \State{$\hspace{1cm}{\rm{\textbf{if}}}\ L_k(g_{k-2,k-1}(x))\neq 0$}
          \State{\hspace{2cm}{List:=List}$-$$ \{k\}$}
    \State{\hspace{2cm} $p_{k-1}(x):=p_{k-2}(x)+\dfrac {y_k-L_k(p_{k-2}(x))}{L_k(g_{k-2,k-1}(x))}g_{k-2,k-1}(x)$}
    \State{\hspace{1cm}{\rm{\textbf{else if\ }}$L_k(g_{k-2,k-1}(x))=\dots=L_{k+s-1}(g_{k-2,k-1}(x))= 0$ and $L_{k+s}(g_{k-2,k-1}(x))\neq 0, s=1,2,\dots,N-k$}}
    \State{\hspace{2cm}$L_k \leftrightarrow L_{k+s}$(Swap the positions of $L_k$ and $L_{k+s}$)}
       \State{\hspace{1cm}{\rm{\textbf{else}}}}
          \State{$\hspace{2cm}{\rm{\textbf{for}}}\ d=k-1, k,\dots,N-1$}
             \State{\hspace{3cm}$g_{-1,d}(x):=x*g_{-1,d}(x)$}
          \State{$\hspace{2cm}{\rm{\textbf{end}}}\hspace{2mm} d $}
       \State{\hspace{1cm}{\rm{\textbf{end\ if}}}}
    \State{\rm{\textbf{end\ while}}}}
	\end{algorithmic}
\end{algorithm}
\newpage

\noindent\textbf{Example 3}
We recalculate Example 2 using the procedure of Algorithm 2.

For the first two interpolation conditions in Example 2, Algorithm 1 allows direct computation of the strongly proper basis $\{1,x\}$.
When computing with the initialized $g_{-1,2}(x)=x^2$, following line 16 of Algorithm 1 we obtain $g_{1,2}(x)=x^2-1$.
Since $g^{(\alpha_3)}_{1,2}(x_{\beta_3})=0$, it does not satisfy the judgment condition, this indicates that $\{g_{-1,0}(x),g_{0,1}(x),g_{1,2}(x)\}$ is not a strongly proper interpolation basis for the interpolation conditions $\{L_1,L_2,L_3\}$. 
Therefore, Algorithm 1 increases the degree of $g_{-1,2}(x)$.
In the framework of Algorithm 2,  if $g^{(\alpha_3)}_{1,2}(x_{\beta_3})=L_3(g_{1,2}(x))=0$, , we need to check  whether  the value of the subsequent interpolation conditions acting on $g_{1,2}(x)$  is equal to zero.
Actually, $L_4(g_{1,2}(x))=\delta_{x_2}\circ D_x(x^2-1)=2 \neq 0$. 
Following lines 20 and 21 of Algorithm 2, we swap the positions of $L_3$ and $L_4$.
Under the new order of interpolation conditions, $\{g_{-1,0}(x),g_{0,1}(x),g_{1,2}(x)\}$ is a strongly proper interpolation basis for interpolation conditions $\{L_1,L_2,L_3\}$.

The final ordering of the interpolation conditions produced by this process is
\begin{displaymath}
    L_1=\delta_{x_0},\ 
L_2=\delta_{x_2},\ L_3=\delta_{x_2}\circ D_x,\ L_4=\delta_{x_1}\circ D_x,
\end{displaymath}   
the Newton-type basis is $\{1,x+1,x^2-1,x^3-x^2-x+1\}$ and the interpolation polynomial is
\begin{displaymath}
    p_3(x)=-2x^3+5x^2+4x-1.
\end{displaymath}

Algorithm 2 demonstrates a key advantage over Algorithm 1 by producing a lower-degree interpolation polynomial in Example 2.
\\ \\
\noindent\textbf{Example 4}
Let $x_0=1$, $x_1=2$, and
the interpolation conditions
\begin{displaymath}
L_1=\delta_{x_0}\circ D_x,\ 
L_2=\delta_{x_1}\circ (1+D_x),\ L_3=\delta_{x_0}\circ (1+D_x^2),\ L_4=\delta_{x_1}\circ (D_x^2+D_x^3). 
\end{displaymath}
Through Algorithm 2, we obtain
\begin{table}[ht]
\centering
\caption{The auxiliary polynomials.}\label{tbl3}
\begin{tabular}{c c c c}
\toprule
\(g_{-1,0}(x)=x\) & \(g_{-1,1}(x)=x^2\) & \(g_{-1,2}(x)=x^3\) & \(g_{-1,3}(x)=x^4\) \\
\midrule
 & \(g_{0,1}(x)=x^2-2x\) & \(g_{0,2}(x)=x^3-3x\) & \(g_{0,3}(x)=x^4-4x\) \\
 &  & \(g_{1,2}(x)=x^3-\frac{11}{2}x^2+8x\) & \(g_{1,3}(x)=x^4-18x^2+32x\) \\
 &  &  & \(g_{2,3}(x)=x^4-6x^3+15x^2-16x\) \\
\bottomrule
\end{tabular}
\end{table}


If four interpolation values $y_1=1,\  y_2=3,\ y_3=2, \ y_4=4$ are given, the interpolation polynomial is
\begin{displaymath}
p_3(x)=\frac{13}{27}x^4-\frac{32}{9}x^3+\frac{98}{9}x^2-\frac{325}{27}x.
\end{displaymath}

\section{Conclusions}\label{sec5}
Due to the discontinuity of the derivative conditions at the interpolation nodes in the Birkhoff interpolation problem, the set of all polynomials satisfying the homogeneous interpolation conditions fails to form an ideal.
Consequently, the computation of the corresponding interpolation basis and the interpolating polynomial becomes more complex.
This paper generalizes the GRPIA and presents two recursive algorithms for computing the strongly proper interpolation basis, the Newton-type basis, and the interpolating polynomial of the Birkhoff interpolation problems. 
For the case where the interpolation conditions consist of monomial differential operators and evaluation functionals, we first reorder the interpolation conditions to obtain an $N-$DOS. 
Then, using the theory of Schur complement, we derive formulas for the corresponding Newton basis and the interpolating polynomial recursively. 
In contrast to existing algorithms that compute the Birkhoff interpolation basis using Gaussian elimination, our recursive approach achieves a substantial reduction in computational cost and storage space.
Simultaneously, Algorithm 1 incorporates a core judgment condition (Line 18 of Algorithm 1) to ensure the non-singularity of the Vandermonde matrix for the Birkhoff interpolation problem. It is possible that in some numerical examples, Algorithm 1 may not yield a minimal-degree interpolation basis.
To obtain lower-degree interpolation basis, we develop Algorithm 2 as an improved version.
Furthermore, Algorithm 2 is capable of solving more general Birkhoff interpolation problems, where the conditions are composed of differential polynomials and evaluation functionals.
\section*{Acknowledgments}
This research was supported by the Department of Education of Jilin Province (JJKH20250468KJ), China.

\end{document}